\documentclass[10pt,oneside,a4paper]{article}

\usepackage{amssymb}
\usepackage{amsmath}
\usepackage{amsthm}
\usepackage{color}
\usepackage{slashbox}
\usepackage{graphicx}
\usepackage{wrapfig}
\usepackage{float}
\usepackage{subfigure}
\usepackage{rotating}
\usepackage{cite}
\usepackage{hyperref}
\usepackage{enumitem}

\newtheorem{thm}{Theorem}[section]

\theoremstyle{remark}
\newtheorem{rem}[thm]{Remark}

\theoremstyle{definition}

\def\eps{\varepsilon}
\def\bu{{\bar{u}}}

\newcommand{\comment}[1]{}
\let\oldmarginpar\marginpar
\renewcommand\marginpar[1]{\-\oldmarginpar[\raggedleft\footnotesize #1]
{\raggedright\footnotesize #1}}

\author{
Jordi-Llu\'is Figueras
\thanks
{
Department of Mathematics, Uppsala University, 
Box 480, 75106 Uppsala, Sweden ({\tt figueras@math.uu.se}).
}
\and 
Marcio Gameiro
\thanks{
Instituto de Ci\^{e}ncias Matem\'{a}ticas e de Computa\c{c}\~{a}o,
Universidade de S\~{a}o Paulo, Caixa Postal 668, 13560-970,
S\~{a}o Carlos, SP, Brazil ({\tt gameiro@icmc.usp.br}).
}
\and 
Jean Philippe Lessard 
\thanks{ 
Universit\'e Laval, D\'epartement de Math\'ematiques et de Statistique, 1045 avenue de la M\'edecine, Qu\'ebec, QC, G1V0A6,
Canada ({\tt jean-philippe.lessard@mat.ulaval.ca}.)}
\and
Rafael de la Llave 
\thanks
{
School of Mathematics, Georgia Institute of Technology, 
686 Cherry St NW, Atlanta, GA 30332, United States ({\tt rll6@math.gatech.edu}).
}
}
\title{A Framework for the Numerical Computation and a Posteriori Verification of Invariant Objects 
of Evolution Equations}

\begin{document}

\maketitle

\begin{abstract}
We develop a theoretical framework for computer-assisted proofs of the
existence of invariant objects in semilinear PDEs.  The invariant objects
considered in this paper are equilibrium points, traveling waves, periodic
orbits and invariant manifolds attached to fixed points or periodic orbits.
The core of the study is writing down the invariance condition as a zero of an
operator.  These operators are in general not continuous, so one needs to
smooth them by means of preconditioners before  classical fixed point theorems
can be applied. We develop in detail all the aspects of how to work with these
objects:  how to precondition the equations, how to work with the nonlinear
terms, which function spaces can be useful, and how to work with them in a
computationally rigorous way.  In two companion papers, we present two
different implementations of the tools developed in this paper to study
periodic orbits.
\end{abstract}

\begin{center}
{\bf \small Keywords} \\ \vspace{.05cm}
{ \small Evolution equation $\cdot$ Periodic Orbits $\cdot$ Contraction mapping \\
Invariant Manifolds $\cdot$ Rigorous Computations $\cdot$ Interval Analysis}
\end{center}

\begin{center}
{\bf \small Mathematics Subject Classification (2010)} \\ \vspace{.05cm}
{ \small 35B32 $\cdot$ 35R20 $\cdot$ 47J15 $\cdot$  65G40 $\cdot$ 65H20 }
\end{center}

\section{Introduction}
This paper establishes the theoretical background  for the computation 
and rigorous verification of different types of invariant solutions for evolution PDEs. 
It is accompanied by two papers \cite{ks_jll_r,ks_jp_marcio} providing independent 
implementations for periodic orbits for the Kuramoto-Sivashinsky PDE.
The choices of spaces, algorithms and implementations are different.
Moreover, paper \cite{ks_jp_marcio} shows how to perform continuation with 
respect to parameters while paper  \cite{ks_jll_r} establishes lower bounds of analyticity
of solutions.

The goal of this paper is to present an approach to the existence of invariant
objects in evolution partial differential equations which can be used to obtain
computer-assisted proofs validating numerical computations. In this paper we
only present some unified strategy that works in several problems. Of course,
in concrete problems one needs to present many more details, but we hope that 
the unified remarks can be used as a systematic guide to understand seemingly 
disparate papers.

We   consider fixed points, periodic traveling waves, periodic solutions and
invariant manifolds attached to them.  Eventually, the study of invariant
manifolds could lead to the study of homoclinic or heteroclinic connections.

The general philosophy is that we formulate the existence of these invariant
objects as the existence of solutions of a functional equation. These
functional equations are shown to have solutions by applying a constructive
version of a fixed point theorem (e.g. contraction mapping in this paper). The
main ingredients of the theorem  is the existence of an approximate fixed
point, as well as some non-degeneracy conditions. One of the byproducts of the
application of such results is local uniqueness.

The approximate solutions can be the product of a numerical computation. It is
then important that one can verify the conditions of the theorem with a finite
computation with finite precision. If these computations are carried out on an
actual computer, we can validate the computation and obtain a
\emph{computer-assisted proof}.

From the point of view of computer-assisted proofs, it is important that the
conditions of the theorem can be verified by a finite computation.  This
entails, in particular, reducing the problem to a finite dimensional problem
and devising theoretically explicit estimates for the truncation errors.  These
studies of the truncation depend very much on the structure of the operators
and the spaces on which we consider them acting. Hence, we  present only some
cases. 

Similar strategies have been already applied in different contexts.
Constructive fixed point theorems were used in the study of Renormalization
Group operators \cite{Collet_Eckmann_Lanford, Eckmann_Koch_Wittwer}, and in
finite dimensional invariant manifolds in
\cite{MR2821596,Mireles_Mischaikow_2013,MR2494688}.

In this paper we  show how to give constructive proofs (computer-assisted) of
existence of some special orbits (fixed points, traveling waves, periodic
orbits and some stable/unstable manifolds) of parabolic equations of the form 
\begin{equation}\label{eq: evolution}
\partial_t u = Lu+N(u), 
\end{equation}
where $L$ is an elliptic operator and $N$ is a semilinear operator (of lower
order than $L$).  Without loss of generality (using translations in $u$ and
redefining $L$), we  assume $N(0) =0$, $DN(0)=0$.  We  make it precise later on
in what sense $DN$ should be understood. In the theoretical development, 
the unknown $u$ could take values in $\mathbb R^n$, even if the main examples 
here are with $n=1$.

This evolution equation is supplemented with boundary conditions. We  impose
boundary conditions by considering equation \eqref{eq: evolution} as defined in
an appropriate Banach space of functions satisfying the boundary conditions.
Indeed the choice of spaces is very important (see Section~\ref{section: Banach
spaces}). Sometimes, we require the existence of a scalar product defined in the
space. {From} the more theoretical point of view, the spaces encode not only
the boundary conditions but also the regularity properties of the solutions
obtained. To perform the analysis, we require that the norm of the spaces is
such that one can study comfortably the operators that appear in the problem
(differential operators, their inverses, etc)  and that they are Banach
algebras under multiplication. To be able to do efficient numerical analysis,
we  also require that the norms of vectors and of linear operators can be
computed accurately and efficiently, and that truncations are easy to estimate.
Some choices of spaces that satisfy these theoretical and practical
requirements are discussed in Section~\ref{section: Banach spaces}.  At some
point, we  find it useful to use the \emph{two spaces approach} of
\cite{Henry81}. 

A concrete example of equation \eqref{eq: evolution} that is worth keeping in
mind and which is used in the concrete numerical implementations, 
see \cite{ks_jp_marcio, ks_jll_r}, is the
Kuramoto-Sivashinsky equation \cite{laqueyetaltri1975, Cohenetaltri1976,
Sivashinksy77, Kuramoto76},
\begin{equation} \label{eq: KS} 
\partial_t u = \alpha \partial_{x}^2  u  +  \partial_x^4 u + u \partial_x u, 
\end{equation}
which we  supplement with periodic boundary conditions in $x$ (possibly also
with the requirement that the solution is odd in $x$). The real number $\alpha$
is a parameter of the problem. In the literature, one can find different forms
of the equation which are equivalent to this one by changing the units of
space, time and $u$. 

Note that the evolution equation \eqref{eq: evolution} is infinite dimensional
and that both $L$ and $N$ are unbounded (in the case of \eqref{eq: KS},
$L(u)=(\alpha\partial_x^4+\partial_x^2)u$ and $N(u)=u\partial_x u$). Hence, one
cannot directly apply the methods of ordinary differential equations, and even
existence of a semi-flow requires methods appropriate for PDEs \cite{Henry81,
Showalter97, Sell_You_book, Robinson_book}.

However, it is worth remarking that, since we are only aiming at obtaining some
particular solutions, we do not need that the evolution is defined nor
well-posed.  Indeed, methods similar to the ones in this paper have been
applied to obtain existence of some specific solutions  in some ill-posed
equations such as the Boussinesq equation and the Boussinesq system
\cite{Boussinesq,Llave08, Fontich_delaLlave_Sire_13,CGL_Boussinesq} 
(numerical computations for some of these problems is work in progress). 
Similar ideas work for delay differential equations, neutral delay equations (with advanced
and retarded delays) and for fractional evolution PDEs. The
study of some specific solutions such as traveling waves, even in ill-posed
equations, has been common-place for a while.

It is important to mention that the methods considered here are not the only
possibilities. Indeed, periodic orbits of equation \eqref{eq: KS} have been
considered in \cite{LanChandreCvitanovic2006, LanCvitanovic2008} and methods
for computer-assisted proofs based on topological methods appeared in
\cite{MR2067140,MR2136516,Piotr_KS_periodic_orbits} and other functional
methods based on propagating and finding fixed points for the time $T$ map
appeared in \cite{ArioliKoch2}. Each method has its own set of practical
difficulties, e.g. propagation methods have difficulties with stiffness,
Fourier discretizations have difficulties when the solutions develop shocks.
All methods share the \emph{curse of dimensionality}.  Hence, it is important
that our methods use only functions of the same number of variables than the
objects they aim to compute.

Let us conclude by mentioning a short list of problems of invariant objects in
PDEs that are not treated in this paper but we believe could be interesting to
attack.  These objects consist of quasi-periodic tori, normally hyperbolic
invariant manifolds, connecting orbits, inertial manifolds and invariant
foliations.

The rest of the paper is organized as follows. In Section~\ref{section: fixed
points}, we present general fixed point theorems for functional equations.  In
Section~\ref{sec:invariance equation}, we reduce the existence of the dynamical
objects to a functional equation and show how to manipulate them in such a way
that they are reduced to a constructive fixed point theorem. In
Section~\ref{section: Banach spaces}, we introduce some basis, spaces and norms
which are useful to perform the analysis. In Section~\ref{sec:guideline}, we
present some guidelines for the implementation of the computer-assisted
verification of the invariant objects following the framework presented in this
paper. 

\section{Fixed point theorems for functional equations}
\label{section: fixed points}

In this section we present general fixed point theorems for functional
equations.  These theorems are useful for the implementation of numerical
schemes and rigorous computer-assisted proofs of the invariance equations for
the different invariant objects treated in this paper. 

Fixed point theorems are classical in the literature. In this section we
present them tailored for the equations we want to work with. Also, in this
section, we present some other tools that we  need while working with the
invariance equations. Sometimes, these equations are underdetermined, so we
impose extra conditions to have local uniqueness of solutions. Also, while
imposing conditions for the attached invariant manifolds, eigenvalue problems
arise, which we transform into a fixed point equation.

\subsection{Lipschitz theorems with preconditioning}\label{subsection: fixed points}

Here we describe some rather general manipulations that are frequently used in
the constructive study of nonlinear equations.  We begin by describing the
manipulations formally and postpone the discussions about ranges, domains and
boundedness of the operators for later. 

We start with an equation 
\begin{equation}\label{eq: functional}
F(u) = 0. 
\end{equation}
Given any injective operator $M$ equation \eqref{eq: functional} is
equivalent to 
\begin{equation}\label{eq: fixed point}
T(u)=u+MF(u)=u.
\end{equation}
\begin{rem}
Notice that we only need that $M$, often called the {\em preconditioner}, is
injective and not boundedly invertible.
\end{rem}

It is important to note that the above manipulations could make sense even when
$F$ and $M$ are unbounded. The only thing that we need is that $M F$ is a well
defined $C^2$ operator.                                                  

A case that we  consider is $F(u)=Lu+N(u)$, where both $L$ and $N$ are
differential operators, hence unbounded. More precisely, $L$ is an invertible
elliptic linear operator and $N$ is another (nonlinear) differential operator.
If the order of $L$ is higher than that of $N$ (the semilinear case), the
operator $L^{-1}N$ is bounded (indeed compact). The choice of $M$ is rather
flexible and there are many choices. In this case $M=L^{-1}$ is suitable, but
sometimes more convenient approximations are considered.

The next two theorems provide sufficient conditions to prove that an operator
is a contraction, and therefore that it has a unique fixed point.  Also, this
theorem provides estimates of the size of the correction term $\delta u$.

\begin{thm}
\label{thm: contraction 2} Consider the operator $F$ in equation \eqref{eq:
functional} defined in $\overline{B_{\rho}(\bu)} = \{ u : \|u-\bu\| \le \rho
\}$, with $\rho > 0$. Denote 
$N_{\bu}$ the nonlinear part of the operator $F$  at the point $\bu$, that is
\[
N_{\bu}(z)=F(\bu+z)-F(\bu)-DF(\bu)z.
\]
Let $B$ be a linear operator such that $BDF(\bu)$, $BF$
and $BN_{\bu}$ are continuous operators.  If, for some $b, K > 0$ we have:

\begin{enumerate}[label=(\alph*),ref=(\alph*)]

\item
\label{item: 1} 
$\|I-BDF(\bu)\| = \alpha < 1$;

\item 
\label{item: 2} 
$\|B\left(F(\bu)+N_{\bu}(z)\right)\|\leq b$ whenever $\|z\| \leq \rho$;

\item 
\label{item: 3} 
$\text{Lip}_{\|z\| \leq \rho} B N_{\bu}(z) < K$;

\item 
\label{item: 4} 
$\frac{b}{1-\alpha} < \rho$;

\item 
\label{item: 5} 
$\frac{K}{1-\alpha} < 1$. 
\end{enumerate}

then there exists $\delta u$ such that $\bu+\delta u$
is in $\overline{B_{\rho}(\bu)}$ and  
is a unique solution of equation 
\eqref{eq: fixed point}, 
with $\|\delta u\| \leq \frac{\|B F(\bu)\|}{1-\alpha-K}$.

\end{thm}

The proof is based on a standard application of the Banach fixed point theorem.
We include it for the sake of completeness.

\begin{proof}
We should check that the map 
\begin{equation}
\label{eq: contracting map}
u\rightarrow \bu-(BDF(\bu))^{-1}B(F(\bu)+N_\bu(u-\bu))
\end{equation}
applies the ball $\overline{B_{\rho}(\bu)}$ to itself and is contracting. A
fixed point of this map is equivalent to a solution of the equation 
\[
F(u+\bu)=0.
\]

Condition \ref{item: 1} implies that 
the linear operator $BDF(\bu)$ is invertible, and 
the norm of its inverse is bounded above by 
$
\frac1{1-\alpha}.
$
Condition \ref{item: 4} implies that the map maps $\overline{B_{\rho}(\bu)}$
to itself, and condition \ref{item: 5} that it is a contraction with Lipschitz 
constant 
$
\frac K{1-\alpha}.
$

Finally, the Banach fixed point theorem assures that there is a unique 
solution $u_*=\bu+\delta u$
in $\overline{B_{\rho}(\bu)}$. The last assertion follows from 
the fact that if we iterate the contracting map \eqref{eq: contracting map}
starting from the point $u_0=\bu$, then
\[
\|\delta u\|=\|u_0-u_*\|\leq \frac{1}{1-\frac K{1-\alpha}}
\|u_0-u_1\|=
\frac{1}{1-\frac K{1-\alpha}}
\|(BDF(\bu))^{-1}BF(\bu)\|\leq \frac 1{1-\alpha-K}\|BF(\bu)\|.
\]

\end{proof}

\subsubsection{The radii polynomial approach}

It may appear at first that Theorem \ref{thm: contraction 2} is not suitable
for numerical computations because is it not readily clear how to find $\rho$.
However there are explicit choices for the constants in the theorem. One
example of how to choose the constants is provided by the classical
Newton-Kantorovich Theorem. Another example, that is a refinement of the
latter, is the \textit{radii polynomial} approach
\cite{MR2338393,MR2776917,MR3077902}.

\begin{thm} \label{thm: contraction 3}
Let $T$ in \eqref{eq: fixed point} be $C^1$, let $\rho >0$ and
$\overline{B_{\rho}(\bu)} = \{ u : \|u-\bu\| \le \rho \}$.  Consider bounds
$\eps$ and $\kappa = \kappa(\rho)$ satisfying 
\begin{align}
\| T(\bu) - \bu \| & \le \eps \label{eq: eps bound} \\ \sup_{w \in
\overline{B_{\rho}(\bu)}} \| DT(w) \|  &\le \kappa(\rho) \label{eq: Z bound}.
\end{align}
If 
\begin{equation} \label{eq: condition contration 3}
\eps + \rho \kappa(\rho) < \rho 
\end{equation}
then $T:\overline{B_{\rho}(\bu)} \rightarrow \overline{B_{\rho}(\bu)}$ is a
contraction with constant $\kappa(\rho)<1$. Moreover, the linear
operator $M$ is injective and therefore $F=0$ has a unique solution in
$\overline{B_{\rho}(\bu)}$.
\end{thm}

\begin{proof}
By \eqref{eq: Z bound}, $\| DT(\bu) \| = \| Id - M DF(\bu) \| \le \kappa(\rho)
< 1$.  Using Neumann series, we conclude that $M$ is injective. Hence, fixed
points of $T$ are in a one-to-one correspondence with the zeros of $F$. Given
any $x \ne y \in B_{\bu}(\rho)$, there exists $z=tx+(1-t)y \in B_{\bu}(\rho)$
with $t \in [0,1]$ such that 
\[
\|T(x)-T(y) \| =  \|DT(z) (x-y) \| \le \| DT(z) \| \| x-y \| < \kappa(\rho) \|
x-y \|.
\]

That shows that $T$ is a contraction. It remains to prove that $T$ maps
$\overline{B_{\rho}(\bu)}$ into itself given $x \in \overline{B_{\rho}(\bu)}$.
\[
\|T(x)-\bu \| \le  \| T(x)-T(\bu) \|  + \|T(\bu)-\bu \| 
\le  \kappa(\rho) \| x-\bu \| + \eps \le \eps + \rho \kappa(\rho) < \rho,
\]
follows by \eqref{eq: condition contration 3}. The contraction mapping theorem
yields a unique fixed point $u_*$ of $T$ in $\overline{B_{\rho}(\bu)}$ and
therefore a unique solution of $F=0$ in $\overline{B_{\rho}(\bu)}$.
\end{proof}

Once the bounds \eqref{eq: eps bound} and \eqref{eq: Z bound} are computed, the
main hypothesis of Theorem~\ref{thm: contraction 3} is to verify the existence
of a radius $\rho>0$ such that inequality \eqref{eq: condition contration 3} is
satisfied.  We may use the notion of radii polynomials to find such $\rho>0$.
The philosophy of the radii polynomial approach is to leave the radius $\rho$
of the closed ball $\overline{B_{\rho}(\bu)}$ variable in case that the
nonlinearities of $T$ are polynomials. 

We can derive analytic estimates (using the fact that the function space is a
Banach algebra) and use interval arithmetic computations to obtain a polynomial
bound for the right hand side of \eqref{eq: Z bound} of the form $\kappa(\rho)
= \kappa_1 +  \kappa_2 \rho +  \kappa_3 \rho^2 + \dots + \kappa_N \rho^{N-1}$, where
the coefficients $\kappa_i$ ($i=1,\dots,N$) are nonnegative and typically $N$ is the degree 
of the polynomial nonlinearity in the problem under investigation.  
The term $\kappa_1$ should be small, thanks to the fact
that $T$ is a Newton-like operator at $\bu$ and the term $\eps$ should be small
if $\bu$ is a good enough numerical approximation.  Then, we can define the
{\em radii polynomial} by
\begin{equation} \label{eq:radii_polynomial}
p(\rho) = \sum_{j=2}^N \kappa_j \rho^j + (\kappa_1-1) \rho  + \eps.
\end{equation}
Let 
\[
\mathcal{I} = \{ \rho > 0 ~:~  p(\rho) < 0 \}.
\]
If the hypothesis \eqref{eq: condition contration 3} of Theorem~\ref{thm:
contraction 3} holds, then $\mathcal{I} \neq \emptyset$. This implies that
$\kappa_1-1<0$, as otherwise we would not be able to find $\rho>0$ such that
$p(\rho)<0$. By Descartes' rule of signs, the radii polynomial
\eqref{eq:radii_polynomial} has exactly two positive real zeros that we denote
by $\rho_-<\rho_+$.  This implies that $\mathcal{I}$ is an open interval.

Heuristically, since we aim at finding small radii, that is $0 < \rho \ll 1$,
then for $\rho$ small $p(\rho) \approx \kappa_2 \rho^2 + (\kappa_1-1) \rho  +
\eps$. If $(\kappa_1-1)^2 - 4 \eps \kappa_2 > 0$, then there should exist an
open interval $\cal I = (\rho_-,\rho_+)$ such that for every $\rho \in \cal I$,
$p(\rho)<0$.
This approach yields a {\em continuum} of
balls of the form $\overline{B_{\rho}(\bu)}$ (with $\rho \in \cal I$) that
contain a unique solution of $F=0$.  The existence of the interval $\cal I$
yields information that may be useful. For instance, if we are interested in
localizing the solutions in the smallest possible ball, then the set
$\overline{B_{\rho_-}(\bu)}$ provides that information. On the other hand, we
may  sometimes be interested in getting a large isolating ball containing the
unique solution. In this case, this information would be given  by the set
$\overline{B_{\rho_+}(\bu)}$. Finally, another information that may be useful
is that the set $\overline{B_{\rho_+}(\bu)} \setminus
\overline{B_{\rho_-}(\bu)}$ does not contain any solution of $F=0$. This
provides a set of non existence of solutions of $F=0$.

\subsection{Extra equations}
\label{subsection: extra conditions}

Across the paper we  encounter several times functional equations
that are underdetermined. 
They have many solutions which are nevertheless physically equivalent 
(e.g. when looking for a parameterization of a periodic orbit, we have the choice of any
point on the orbit to correspond to $t=0$ in the parameterization).
So, extra normalizing conditions are needed to obtain a unique solution.
There are several alternatives that can be 
chosen.  Some of them are: 
\begin{itemize}
\item \emph{Fixed point in phase space}: This extra condition 
arises when we force the solution $u$ to lie on a codimension 1 subspace. 
In the evolution 
language, this is equivalent to fixing a Poincar\'e section. 
This condition is stated 
in algebraic terms as 
\begin{equation}
\langle u, v\rangle = 0 
\end{equation}
for some $v$.
Note that we require the scalar product $\langle \cdot, \cdot \rangle$ to be 
a differentiable operator in the Banach space we are considering. 
For example, 
the $L^2$ scalar product satisfies this property in the space 
of differentiable functions. See
Section \ref{section: Banach spaces} 
for a detailed treatment of this fact.                            

\item \emph{Fixed phase}: This extra condition impose that the correction 
$u-\bu$ is perpendicular to the initial guess $\bu$. This is done 
by imposing that 
\begin{equation}
\langle u-\bu, \bu\rangle = 0.
\end{equation}
This condition is useful when we want to eliminate a translation.

\end{itemize}

\subsection{A general eigenvalue-eigenvector equation}
\label{subsection: eigenvalue}

During the exposition we  encounter several times eigenvalue-eigenvectors equations of
the form
\begin{equation}\label{eq: eigenvalue problem}
Lv+Av = \lambda v, 
\end{equation}
where $L$ and $A$ are linear operators such that $L^{-1}A$ is a bounded operator.
This equation is underdetermined, but 
after imposing an extra scalar equation (for example $\langle v, v \rangle = 1$), it 
can be transformed into a fixed point equation for the unknowns $(\lambda, v)$. In more
concrete terms, after adding this last condition we obtain that the pair 
$(\lambda, v)$ satisfies the system of equations
\begin{equation}\label{eq: eigen system}
F(\lambda,v)=
\left(
\begin{array}{cc}
Lv+Av -\lambda v\\ 
\langle v, v \rangle -1
\end{array}
\right)=0.
\end{equation}
In all problems in this paper, the linear operator $L$ is easily invertible,
and satisfies that both $L^{-1}$ and $L^{-1}A$ are continuous operators. Following the formulation 
as in Subsection \ref{subsection: fixed points}, we multiply equation 
\eqref{eq: eigen system} by the operator 
\[M=
\begin{pmatrix}
L^{-1} & 0\\
0 & 1\\
\end{pmatrix}.
\]

\begin{rem}
Notice that the nonlinear (in fact 
quadratic) part of the last system is easily treatable.
\end{rem}

\begin{rem}
In the literature the extra condition in terms of norms is classic, see for
instance \cite{bookKato}. Other choices are also possible. For example, in
\cite{MR3204427} the authors use the condition of fixing the value of one
of the coefficients.  Again, we need that the extra constraint is
differentiable in the Banach space considered. We also need that the condition
given by the extra equation leads to a good preconditioning with (hopefully)
not too large norm.
\end{rem}

\section{Invariance equations and formulation as a fixed point} \label{sec:invariance equation}

In this section we reduce the existence of the dynamical objects to functional
equations and show how to manipulate them in such a way that they are reduced
to a constructive fixed point problem.  The manipulations are largely formal
(we  ignore issues such as the domain of unbounded operators and assume
invertibility when needed, etc). These are addressed when we specify the
spaces.

Unstable manifolds attached to invariant objects, due to the dissipative
character of the evolution equation \eqref{eq: evolution}, are finite
dimensional. In this paper we  present only the case of one-dimensional
attached manifolds. To see how to deal with higher dimensional ones the reader
should consult \cite{Cabre_Fontich_delaLlave_05}.

Along all the presentation of all the problems, we  follow the notation in
Subsection \ref{subsection: fixed points}, we  specify the operator $F$ and the
linear preconditioner $M$. If necessary, we  discuss the nonlinear term
$N_{\bu}$.

\subsection{Equilibrium points}
\label{sec:equilibrium}

In this section, we show how to use our framework to compute equilibrium points
of an evolution equation.  The example to keep in mind is that of a parabolic
PDE.  In this case, the equilibrium point is the solution of an elliptic PDE.
There is an extensive literature on the a posteriori verification of solutions
of elliptic PDE, notably for finite element discretizations (see e.g.
\cite{MR2019251,MR1810529,MR1849323,MR2109916} and the references therein).
The presentation in this section is more geared towards spectral
discretizations, which is also used for the other objects we  discuss later,
and to obtain computer-assisted proofs in very smooth norms.

A function $u$ is an equilibrium (time independent solution) of \eqref{eq:
evolution} if
\begin{equation}
\label{eq: equilibrium}
F(u) = Lu+N(u) = 0.
\end{equation}
In our applications, $L$ is an elliptic operator, $N$ is a nonlinear
differential operator of smaller order than $L$, and $L^{-1}N$ is a compact
operator when considered acting on appropriate spaces. Hence, the operator $M$
considered in Subsection \ref{subsection: fixed points} is $-L^{-1}$.  A
discussion of suitable spaces appears in Section~\ref{section: Banach spaces}.
If the operator is compact then finite dimensional approximations 
have a good chance of working.

In some cases, the operator $L$ is not invertible but $L+\lambda I$ is (this
happens for instance for elliptic operators and for most $\lambda\in\mathbb
C$).  Moreover, adding the parameter $\lambda$ may be numerically advantageous
even if $L$ is invertible since the operator $MF + I$ may be more numerically
stable, with $M = (L + \lambda I)^{-1}$.  Indeed, $(L + \lambda I)^{-1}$ may be
easier to reliably compute than $L^{-1}$.

In practical applications, further preconditioning by finite dimensional
matrices may be advantageous. Specially when the operator $L^{-1} N$ is
compact. 

\subsection{Traveling waves}

Traveling waves are solutions of \eqref{eq: evolution} of the form $u(x, t) =
\phi(x-ct)$ ($x$ can be $n$ dimensional but the displacement $c$ is one
dimensional). Hence, to find a traveling wave (using that the operators $L$ and
$N$ are time independent) we obtain an operator equation for $c$ and $\phi$ of
the form
\begin{equation}
\label{eq: traveling wave}
c\phi' +L\phi+N(\phi) = 0.
\end{equation}
This equation is underdetermined, since leaving $c$ unmodified and shifting
$\phi$ along $x$ also gives a solution of \eqref{eq: traveling wave}. So we add
the extra scalar condition $\langle \phi, \phi \rangle = 1$ to obtain the
system of equations
\begin{equation*}
F(\phi,c) = 
\left(
\begin{array}{cc}
c\phi' +L\phi+N(\phi) \\
\langle \phi, \phi \rangle -1
\end{array}
\right) =0.
\end{equation*}
For this equation we obtain that the preconditioner is
\[
M = 
\begin{pmatrix}
(L+c\partial_x)^{-1} & 0\\
0 & 1
\end{pmatrix}.
\]

\begin{rem}
In order to obtain a complete formulation, one should specify the space of
functions to which the function $\phi$ belongs. 

In practice, this method works only for periodic traveling waves.  Clearly,
Dirichlet boundary conditions or odd periodic conditions do not allow
non-trivial traveling waves.  

The problem of traveling  waves on the whole line with conditions at infinity
are not considered here. The study of these \cite{Smoller_book} is
mathematically very challenging since the spectral properties of the
linearization are difficult (e.g. one has continuous spectrum). Moreover, this
study could be conducted as homoclinic orbits
\cite{Mireles_Mischaikow_2013,MR2821596}.
\end{rem}

\begin{rem}
When the space variable $x$ is one-dimensional, the study of existence of
traveling waves reduces to the study of periodic solutions of the one parameter
family of differential equations
\[
c\phi' +L\phi+N(\phi) = 0,
\]
and hence ODE techniques can be applied.
\end{rem}

\subsection{Periodic orbits}

Periodic orbits are solutions of equation \eqref{eq: evolution} of the form
$u(x, t)=u(x, t+T)$, with $T > 0$. The period $T$ is also an unknown of the
problem. To write down the functional equation $(u, T)$ satisfy, we parametrize
them by $v(x, \theta) = u(x, \frac\theta a)$, obtaining  
\begin{equation}\label{eq: periodic orbit}
a\partial_\theta v = Lv+N(v), 
\end{equation}
where $0\leq \theta \leq 2\pi$ and $v(x, \theta+2\pi)=v(x, \theta)$.

Notice that the frequency $a$ is now another unknown of the problem equivalent
to $T$. Hence, we proceed again as in the traveling wave case and add an extra
scalar equation.  The most suitable one is fixing the phase of the
parameterization. 

More concretely, given an approximate solution $(\bar{a}, \bar{v})$ of 
equation \eqref{eq: periodic orbit}, we look for a correction $(\alpha, w)$ satisfying 
the system of equations
\begin{equation}
\label{eq: system periodic orbit}
F(\alpha,w)=
\left(
\begin{array}{cc}
(\bar{a}+\alpha) \partial_\theta (\bar{v} + w) - L(\bar{v} + w)-N(\bar{v} + w)  \\
\langle \partial_\theta \bar{v}, w \rangle  
\end{array}
\right)=0.
\end{equation}

A good preconditioner for equation \eqref{eq: system periodic orbit} is 
\[
M=
\begin{pmatrix}
(\bar{a} \partial_\theta-L)^{-1} & 0\\
0 & 1
\end{pmatrix}.
\]

\subsection{One dimensional manifolds attached to fixed points} \label{sec:1Dmanifold_to_fixed_pt}

Given a fixed point $u_0$ that satisfies 
equation \eqref{eq: equilibrium}, we are seeking 
a function $U(s,x)$ that satisfies
\begin{equation}\label{eq: attached fixed point}
\lambda s \partial_s U(s,x)  = L U(s,x) + N(U(s,x)),
\end{equation}
where $L$ and $N$ are differential operators in $x$.

It is clear that if we find  a 
solution that satisfies 
\eqref{eq: attached fixed point}, 
then
\begin{equation}\label{eq:evolutionparameter}
u(t,x) = U(e^{\lambda t}, x) 
\end{equation}
satisfies the evolution equation  \eqref{eq: evolution}.

We can think geometrically of $U$ as providing a parameterization of a one
dimensional manifold invariant under the semiflow. For each value of $s$, $U(s,
\cdot)$ is a point in a function space.  Hence, as $s$ traces an interval, we
obtain a line segment in the function space.  The dynamics of the evolution
equation \eqref{eq: evolution} moves along the curve with exponential dynamics.

At first \eqref{eq: attached fixed point} may looks similar to 
equation \eqref{eq: equilibrium} for functions 
of two variables. We can indeed write it 
as 
\[
\tilde L U + N(U) = 0 
\]
where $\tilde L = - \lambda s \partial_s + L $. 
Nevertheless, there are two 
important differences: 
\begin{itemize}
\item 
We emphasize that in \eqref{eq: attached fixed point}, the unknowns are both
the (real) number   $\lambda$ and the function of two variables $U$. 
\item 
We also note that the equation  \eqref{eq: attached fixed point} does not admit
a unique solution. If $(\lambda, U) $ is a solution of \eqref{eq: attached
fixed point}, and $\sigma$ is any number, then $(\lambda, U_\sigma)$, with
$U_\sigma(s, x) = U(s\sigma, x)$, is also a solution of \eqref{eq: attached
fixed point}.
\end{itemize}

Despite the above important differences, we are able to reduce the problem to a
fixed point problem, but we  need some extra steps.  Our first goal is to use
the techniques already developed in Sections~\ref{subsection: eigenvalue} and
\ref{sec:equilibrium} to obtain approximate solutions of \eqref{eq: attached
fixed point}. In particular, we  have to find  $\lambda$.  Then,  we  show how
to start a fixed point algorithm based on these approximate solutions. See the
following subsections for a discussion of the method.

More detailed and more general  results along these lines appear in
\cite{delaLlave97,Cabre_Fontich_delaLlave_03,Cabre_Fontich_delaLlave_05}. In
particular,  there are results for manifolds of dimension higher than one.

\subsubsection{Approximate solutions of the parameterization
equation up to first order} 

Proceeding heuristically for the moment, we seek 
$U$ as a power series in $s$ with coefficients which are functions of
$x$
\[
U(s,x) = \sum_{n \ge 0} u_n(x) s^n.
\]
We also assume that the operator $N$ is differentiable. Of course, all this 
needs to be verified later with the appropriate definitions of Banach spaces, 
but in this section, we just describe the numerical algorithm. 

In a first step, we obtain that $u_0$ should 
satisfy the equilibrium  equation \eqref{eq: equilibrium} discussed in Section~\ref{sec:equilibrium}. 
In the second step, we obtain that the coefficient  $u_1$  and $\lambda$ satisfy
\begin{equation}\label{eq:firstorder}
\lambda u_1 = L u_1 + DN(u_0) u_1 .
\end{equation}

Equation \eqref{eq:firstorder}  is precisely an eigenvalue equation of the type discussed in 
Section~\ref{subsection: extra conditions}. 
Picking an eigenvalue determines uniquely the manifold to be computed. While 
\eqref{eq:firstorder} determines $\lambda$ uniquely, $u_1$ is only determined up to a multiple. 

As it turns out, this choice of the multiple of the eigenvector corresponds 
to the non-uniqueness pointed out in Section~\ref{subsection: extra conditions}. Once we fix
the length of the eigenvector, the (local) uniqueness is settled. This can be used as
a normalization to obtain uniqueness. 

\subsubsection{Approximate solutions of the parameterization
equation to order higher than one
}
\label{sec:resonances}

If we continue solving equation \eqref{eq: attached fixed point}
to orders (in $s$) higher than one, we see 
that equating the coefficients of $s^n$ on both sides of \eqref{eq: attached fixed point}
we obtain
\begin{equation}\label{ordern} 
n \lambda u_n = (L + DN (u_0) )u_n  + R_n(u_0,u_1, \ldots, u_{n-1}),
\end{equation} 
where $R_n$ is an expression (polynomial in $u_1,\ldots , u_{n-1}$) 
which can be readily computed substituting the expansion already computed in the
Taylor expansion of $N$. See \cite{Haro_survey} for a detailed exposition 
on the numerical aspects.

If 
\begin{equation} 
\label{eq:nonresonance} 
n \lambda \notin {\rm Spec}(L + DN(u_0)),
\end{equation}
 one can find $u_n$ and proceed. On the other hand, if \eqref{eq:nonresonance}
 fails and $n \lambda \in {\rm Spec}(L + DN(u_0))$, it is not hard to construct
 examples of non-linearities $N$ for which the invariance equation \eqref{eq:
 attached fixed point} has no solution.

For the Kuramoto-Sivashinsky equation \eqref{eq: KS}, the spectrum of $L +
DN(u_0)$ accumulates to $-\infty$, hence a negative $\lambda$ is rather
problematic.  On the other hand, positive $\lambda$ are tractable because the
spectrum of $L + DN(u_0)$ only contains a finite number of positive
eigenvalues.  In particular, the most unstable eigenvalue automatically
satisfies the non-resonance condition \eqref{eq:nonresonance}. 

\subsubsection{Preconditioning the fixed point equation}

Given the periodic orbit $u_0$ and the eigenpair $(\lambda, u_1)$, we can write
the functional equation \eqref{eq: functional} for the high order
$s^2g(x,s)=U(x,s)-u_0(x)-u_1(x)s$ as
\begin{equation*}
F(g)=\lambda s \partial_s (u_0+u_1 s+ g s^2) - L(u_0+u_1s+g s^2)-N(u_0+u_1s+g
s^2) = 0.
\end{equation*}
Then, we  apply the existence theorems of Section \ref{section: fixed points}
with the approximation of $g$ computed using the method of
Section~\ref{sec:resonances} and with preconditioner 
\[
M= (\lambda s \partial_s-L)^{-1}.
\]

\begin{rem}
As stated in Section~\ref{sec:resonances}, the existence of the
parameterization of the invariant manifold depends on the absence of resonances
in \eqref{eq:nonresonance}.  This absence is equivalent to finding the
operator $B$ required in Theorem \ref{thm: contraction 2}. In general, the
proof of absence of resonances is a delicate question, since we should prove
that $n\lambda$ does not accumulate in the spectrum. There are however easier
cases, for instance in the case of the fast (strong) unstable manifold.
\end{rem}

\subsection{One dimensional invariant manifolds attached to periodic orbits}

Similar to the case of invariant manifolds attached to fixed points, manifolds
attached to periodic orbits are parameterized by $U(x, \theta, s) = u_0(x,
\theta)+\sum_{n\geq 1}u_n(x, \theta)s^n$, where $u_0(x, \theta)$ is the
parameterization of the periodic orbit, hence satisfying equation \eqref{eq:
periodic orbit}. The functional equation for $U(x, \theta, s)$ is 
\begin{equation}\label{eq: attached periodic orbit}
a \partial_\theta U(x, \theta, s)+\lambda s \partial_s U(x, \theta, s) = L U(x,
\theta, s)+N(U(x, \theta, s)).
\end{equation}

As in Section~\ref{sec:1Dmanifold_to_fixed_pt}, the eigenpair $(\lambda, u_1)$
satisfies a linear equation, 
\begin{equation*}
a \partial_\theta u_1+\lambda u_1 = L u_1+DN(u_0)u_1,
\end{equation*}
and it is solved as explained in Section \ref{subsection: eigenvalue}. 

Then, once the eigenpair $(\lambda, u_1)$ is determined, the term
$ g(x,\theta, s)=\sum_{n \geq 2} u_n(x, \theta) s^{n-2} $
satisfies the functional equation 
\begin{equation*}
F(g)=
a \partial_\theta (u_0+ u_1s+ g s^2)+
\lambda s \partial_s (u_0+u_1s+ g s^2 )
- L 
(u_0+ u_1 s +g s^2) 
-N(u_0+ u_1s+ g s^2) = 0.
\end{equation*}
This equation is then preconditioned by multiplying it 
by the linear operator 
\[
M=(a \partial_\theta+\lambda s \partial_s-L)^{-1}.
\]

Again, we refer to \cite{delaLlave97, Cabre_Fontich_delaLlave_03,
Cabre_Fontich_delaLlave_05} for results on higher dimensions in this case.

\section{Functional analysis considerations}
\label{section: Banach spaces}

In this section we present some spaces which are useful to perform the
analysis. The choice of spaces is dictated by the need to have effective
estimates for the operators we are considering.  We want to make explicit the
considerations that lead to the choices and provide several elementary lemmas.
There are other choices which are also possible which may be preferable if the
considerations change. We  also present some of these possible choices and,
indeed present some of the shortcomings of the different choices. 

In the exposition we consider the simplest case of functions of one variable.
The case of functions of several variables can be reduced to this one by
expanding the functions in Fourier (or Taylor) series with respect one of the
variables and considering the coefficients of this expansion as functions of
the other variables. This process is recursive and very easy to implement 
in computer languages which support overloading. 

\subsection{Choice of basis }
We consider expansions in Fourier (or Taylor) series of eigenfunctions
$\phi_n(x)$ of $L$. In practice Fourier series of functions satisfy the
boundary and symmetry conditions.  For our purposes these are useful when
dealing with differential operators. Indeed, since the constant coefficients
differential operators are diagonal it is easy to estimate truncations. 

Recent years have witnessed several successful attempts in combining Fourier
series with computer-assisted estimates to prove existence of solutions of
finite and infinite dimensional dynamical systems. Rigorous results about
periodic solutions of ODEs, delay equations and PDEs
\cite{MR2166710,MR2592879,Piotr_KS_periodic_orbits,MR2788972,HLM, ArioliKoch2},
dynamics of infinite dimensional maps \cite{MR2067140,MR3124898}, equilibria of
PDEs \cite{MR2470145,MR2679365,MR2718657,MR2776917,MR3077902,MR1699017} and
global dynamics of parabolic PDEs \cite{MR2136516,MR2441958} have been obtained
using Fourier series. By now there are many efficient implemented algorithms
handling Fourier series. In fact, there is specialized hardware (GPU) carrying
out operations with Fourier coefficients. The main (well known) shortcoming of
Fourier series is that they are not adaptive and that they are computationally
hard if one is dealing with phenomena that present singularities or high
oscillations in small regions of the phase space.

There is an extensive literature on the a posteriori verification of solutions
of elliptic PDE, notably for finite element discretizations
\cite{Ainsworth_Oden_2000, Kim_Park_2010, Carstensen_Merdon_2011, Ern_2011,
Watanabe_Kinoshita_Nakao_2013b,
Watanabe_Kinoshita_Nakao_2013,MR2019251,MR1810529,MR1849323,MR2109916}.  Among
them, methods involving rigorous enclosure of eigenvalues of nonlinear
operators are used in \cite{MR2019251,MR1810529} to prove existence of
solutions of second-order elliptic boundary value problems.  In
\cite{MR1849323,MR2161437}, a technique combining the Schauder fixed point
theorem and a priori error estimates for finite element approximations in
Sobolev spaces is applied to prove the existence of solutions of elliptic
problems. 

Splines interpolations have also led to rigorous existence results for PDEs
\cite{JavierGomez-Serrano}, boundary value problems \cite{MR1639986},
connecting orbits \cite{MR2821596,MR3207723,MR3207723} and localized solutions
of reaction-diffusion PDEs \cite{vdBGW}.

There are, several works on theoretical estimates \cite{Cohen_book} or on
numerical work with wavelets \cite{Wickerhauser_2002} but we are not aware of
rigorous computer-assisted estimates. 

Interpolation with Chebyshev polynomials has received a great deal of attention
\cite{Chebfun} and there are computer-assisted estimates
\cite{Lessard_Reinhardt_2014,correc_lessard,BDLM}. Note that Chebyshev
polynomials are closely related to Fourier series. 

Taylor methods have led to rigorous integrators of flows of ODEs
\cite{MR1652147,MR1962787,MR2312391,MR1930946} and rigorous enclosure of
invariant manifolds \cite{Mireles_Mischaikow_2013,JDMJ01,mirelesCNSNS_2014}
leading to effective computation of connecting orbits.

\subsection{Choices of norms} 

We  consider spaces of functions (indexed by the parameter  $\mu$)
\[
u(x)=\sum a_n \phi_n(x)
\]
with 
\[
\|u\|_\mu = \sum |a_n| W_n^\mu
\]
for conveniently chosen weights $W_n^\mu$. We say that 
this space is a {\emph weigthed $\ell^1$ norm.}

There are several reasons to take weighted $\ell^1$ spaces rather than other
spaces such as  weighted $\ell^2$, which have interesting properties such as
being a Hilbert space or other $\ell^p$ spaces, which have nice properties such
as being reflexive Banach spaces. 

\begin{itemize}
\item
The norms can be evaluated more reliably in finite precision computation
and in interval arithmetic than those involving higher powers. 

The main reason is that the round off is much more dramatic when we are adding
modes of different sizes. Since the squares of numbers have a bigger spread of
sizes than the numbers themselves, adding numbers is less affected by rounding
than adding their squares. 

\item
Sensitivity to round  off is not the only numerical criterion. 
Note that while the $\ell^\infty$ norm is insensitive to the roundoff, it is
insensitive to many changes in the function. 

\item
It is easy to compute numerically the norm of an operator. 

As we  see in \eqref{operatornorm}, the norm of a matrix of size $N$ can be
estimated  rigorously in $N^2$ operations. (The estimates are accurate up to
round off). 

\item 
The space $\ell^1$ is not reflexive, but at least its dual $\ell^\infty$ is
explicit and rather manageable.
\end{itemize}

Some weights that we  consider are of the form 
\[
W_n^{\mu_1, \mu_2} = (1+|n|)^{\mu_1} e^{\mu_2|n|}
\]
for $\mu_1, \mu_2 \geq 0$. Note that the case $\mu_2 > 0$ automatically 
ensures analyticity of the solutions, while the case $\mu_1 > d$ ensures 
that the functions are $C^d$.

With these norms, operator norms are easy to compute. Given an operator $T$
with coefficients $T_{n, m}$ associated to the base $\left\{\phi_n\right\}$,
its norm is  

\begin{equation}
\label{operatornorm} 
\|T\|_{\mu\rightarrow \mu'} = 
\sup_{m} \dfrac{\sum_{n}
|T_{n,m}|W_{n}^\mu}{W_m^{\mu'}}
.
\end{equation}
When the target space coincides with the definition space, we  denote 
$\|T\|_{\mu\rightarrow \mu}$ by $\|T\|_\mu$.

Sometimes there is the need to compute the norm of the composition of two
operators, let us say $A$ and $B$, such that $A$ is a compact operator and $B$
is unbounded, but with $AB$ bounded. Suppose that the norm of both operators
can be computed but the norm of their composition cannot. Then, it is useful to
use \emph{the two spaces approach} appearing in \cite{Henry81}. This consists
on bounding the norm $\|AB\|_\mu$ by the product of the norms
$\|A\|_{\mu'\rightarrow\mu}$ and $\|B\|_{\mu\rightarrow\mu'}$.

Another space that we consider for the solutions 
\[
u(x) = \sum a_n \phi_n(x)  
\]
is $\ell^\infty$ with a weighted norm. More precisely, given a sequence $a = \{
a_n \}_{n}$ and a decay rate $s > 1$, we define the norm
\[
\| a \|_s = \sup_{n} \left\{ |a_n| \omega_n^s \right\},
\]
where the weight $\omega_n^s$ is defined by 
$\omega_{n_i}^{q_i}$, with
\[
\omega_{n}^{s} =
\begin{cases}
~1, & \text{if} ~ n = 0 \\
| n |^{s}, & \text{if} ~ n \neq 0.
\end{cases}
\]

Using this norm we define the Banach space
\[
X = \{ a : \| a \|_s < \infty \},
\]
of algebraically decaying sequences with decay rate $s$. The Banach space
$(X,\| \cdot \|_s)$ with $s>1$ is an algebra under discrete convolution, but it
is not directly a Banach algebra. However, it is possible to use analytic
estimates to obtain $C=C(s)$ such that $\| a * b\|_s \le C(s) \| a \|_s \| b
\|_s$ \cite{MR3125637,MR3077902}. With the new norm $\| \cdot \|_X = C(s) \|
\cdot \|_s$, $(X,\| \cdot \|_X)$ is a Banach algebra under discrete
convolutions.

\section{Implementation of the verification method}
\label{sec:guideline}

In this section we present the general ideas for the implementation of the
computer-assisted verification of the invariant objects, that is, how to solve
in practice problem \eqref{eq: functional}.  Usually the need for proving the
existence of an invariant object comes from the fact that there is evidence of
its existence. This evidence could come from different sources. For example,
some non-rigorous numerical computations, or heuristic perturbative arguments.
Thus, we have an approximation of the object of interest and we want to give a
rigorous proof of its existence.  The main steps for the proof of the existence
of these objects are the following: 

\begin{itemize}
\item Compute a good approximation $\bu$ of the 
invariant object, that is, an approximate solution of \eqref{eq: functional}.
The proof to be produced assures the existence of a nearby invariant object,
with explicit estimates in terms of norms of how close it is to the
approximation $\bu$.

\item Construct the preconditioned equation \eqref{eq: fixed point}. 
The preconditioner $M$ is problem dependent, so we should detect the operator
$L$ that must be preconditioned ($M$ will be its inverse). As stated in all the
cases, the inverse of this operator is easily stated, since it is expected that
it has the form of $P(D_1, \cdots, D_n)$, where $P$ is a polynomial with
constant coefficients and $D_i$ are differential operators of the form
$\partial_x$, $\partial_\theta$, and so on.

\item Once the preconditioned equation \eqref{eq: fixed point} is constructed, 
we apply one of the fixed points Theorems \ref{thm: contraction 2} or 
\ref{thm: contraction 3}. Choosing which one must be used is more a matter 
of taste. Both are equivalent and provide similar qualitative information 
about the fixed point.

\item Finally, for the application of the fixed point theorem one must evaluate
the quantities required in the theorems. These quantities are bounds on the
error of the approximation $F(\bu)$, on the derivative of $T$ and on the
Lipschitz constant of the derivative of $T$ around the approximation. The
evaluation of these three quantities are problem dependent, and must be treated
in a case by case basis. 
\end{itemize}

In general, one must split the problem into two subproblems, a finite
dimensional one and an infinite dimensional one. This splitting is done in
accordance to how the functional equation deals with the finite dimensional
part and the infinite dimensional part.  The finite dimensional part is handled
with the aid of a computer, while the infinite dimensional part must be dealt
with by hand. Although a priori this may sound very challenging, this is
actually a treatable problem. The action of the contraction operator on the
modes of the infinite dimensional part can be managed since, hopefully, the
linearization of the operator is compact. The choice of truncations seems
to be different depending on the problem.

\section{Concluding remarks}
\label{sec:conclusion}

In this paper we present rather general methods for the computation and
rigorous verification of several invariant objects arising in the study of
PDEs.  We specify the functional equations that these invariant objects must
satisfy, and describe how them can be rigorously computed. All these functional
equations have the property that the linearization around their zeros have a
spectral gap around zero. This allows us to set up a contraction mapping problem
that leads to the rigorous verification of the computed objects.

In two companion papers we present two different implementations of the method
presented in this paper to rigorously verify the existence of time periodic
orbits for the Kuramoto-Sivashinsky equation~\eqref{eq: KS}.

\section*{Acknowledgments}
J.-L. F. was partially supported by Essen, L. and C.-G., for mathematical
studies.  M.G. was partially supported by FAPESP grants 2013/07460-7,
2013/50382-7 and 2010/00875-9, and by CNPq grants 305860/2013-5 and
306453/2009-6, Brazil.  R. L. was partially supported by NSF grant DMS-1500493.
We are also greatful for useful discussions to G. Arioli, P. Cvitanovic, J.
Gomez-Serrano, A.  Haro, H. Koch, K.  Mischaikow, W. Tucker and P. Zgliczynski.

\bibliography{./bibliography}{}
\bibliographystyle{abbrv}

\end{document}